\documentclass[final,3p,times]{elsarticle}

\usepackage[utf8]{inputenc}
\usepackage{siunitx}
\usepackage{amsmath}
\usepackage{array}
\usepackage{amsthm}
\usepackage{amsfonts}
\usepackage{graphicx}
\usepackage{float}
\usepackage{bm}
\usepackage[final]{microtype}
\usepackage{mathtools}
\mathtoolsset{showonlyrefs}

\usepackage{todonotes}

\usepackage{hyperref}
\hypersetup{
	hidelinks,
    colorlinks=black,
    linkcolor=black,
    filecolor=black,      
    urlcolor=black,
}

\newtheorem{theorem}{Theorem}
\newtheorem{corollary}{Corollary}
\newtheorem{assumption}{Assumption}
\newtheorem{definition}{Definition}
\newtheorem{example}{Example}
\newtheorem{lemma}{Lemma}

\newtheorem{proposition}{Proposition}
\newtheorem{remark}{Remark}

\newcommand{\abs}[1]{\left\lvert#1\right\rvert}

\newcommand{\barpow}[1]{\left\lfloor#1\right\rceil}
\newcommand{\sign}[1]{\mbox{sign}(#1)}
\newcommand{\Real}[1]{\mathrm{Re}(#1)}
\newcommand{\Imag}[1]{\mathrm{Im}(#1)}

\def\Red#1{\textcolor{red}{#1}}

\journal{ArXiv}

\begin{document}

\begin{frontmatter}



\title{On the design of non-autonomous fixed-time controllers with a predefined upper bound of the settling time
\footnote{\Red{This is the preprint version of the accepted Manuscript: D. Gómez-Gutiérrez, “On the design of non-autonomous fixed-time controllers with a predefined upper bound of the settling time”, International Journal of Robust and Nonlinear Control, 2020, ISSN: 1099-1239. DOI. 10.1002/rnc.4976.
Please cite the publisher's version. For the publisher's version and full citation details see:
\url{https://doi.org/10.1002/rnc.4976}.
}}
}

\author[label0,label1]{David~Gómez--Gutiérrez}
\ead{David.Gomez.G@ieee.org}

\address[label0]{Multi-agent autonomous systems lab, Intel Labs, Intel Tecnología de M\'exico, Av. del Bosque 1001, Colonia El Bajío, Zapopan, 45019, Jalisco, M\'exico.}
\address[label1]{Tecnologico de Monterrey, Escuela de Ingenier\'ia y Ciencias, Av. General Ram\'on Corona 2514, Zapopan, 45201, Jalisco, M\'exico.}

\begin{abstract}
This paper aims to introduce a design methodology to stabilize a chain of integrators in a fixed-time with predefined Upper Bound for the Settling-Time (\textit{UBST}). This approach is based on time-varying gains (time-base generator) that become singular as the time approaches the predefined convergence time. We present the conditions such that every nonzero trajectory convergence exactly at the predefined time with feedback laws that are simpler than in previous approaches. Additionally, we present results such that the origin is reached before the singularity occurs, making this approach realizable in practice. A significant contribution, since current results, based on time-varying gains, required the time-varying gain to tend to infinity as the time approaches the prescribed/predefined-time. 

Moreover, our approach guarantees fixed-time convergence with predefined \textit{UBST} even in the presence of bounded disturbances (with known bound) and provides a methodology to re-design an autonomous fixed-time controller where the estimate of the \textit{UBST} for the closed-loop system is very conservative into a less conservative one with predefined \textit{UBST} that can be set arbitrarily tight. We present numerical examples to illustrate the effectiveness of the approach together with comparisons with existing autonomous fixed-time controllers.
\end{abstract}


\end{frontmatter}

\section{Introduction}
Recently, toward the design of closed-loop systems satisfying time constraints, there has been increasing interest in the control community on a class of finite-time dynamical systems where there exist an Upper Bound of the Settling Time (\textit{UBST}) which is independent on the initial conditions, such systems are known a fixed-time systems~\cite{Polyakov2012a,Sanchez-Torres2018,Andrieu2008,Basin2019}.

To apply such results in scenarios of time constraints, there has been some effort in deriving controllers defining apriori the \textit{UBST} as a parameter of the system~\cite{Sanchez-Torres2018,Aldana-Lopez2018,aldana2019design}. On the one hand, autonomous controllers have been derived based on Lyapunov analysis~\cite{Polyakov2012a,Aldana-Lopez2018,Sanchez2019,Sanchez2019IJC,Munoz2019,Basin2016Finite-andTechniques}. However, even if non conservative estimates of the \textit{UBST} are obtained in the scalar case~(see e.g. \cite{Sanchez-Torres2018,Aldana-Lopez2018}), the estimate of the \textit{UBST} becomes too conservative in high-order systems (see, e.g. the example in Section~5 of \cite{Basin2016ContinuousRegulators}\footnote{In~\cite{Seeber2019} it was shown that the proof of~\cite[Theorem~1]{Basin2016ContinuousRegulators} had flaws. Since the same arguments of~\cite[Theorem~1]{Basin2016ContinuousRegulators} were used in~\cite[Theorem~2]{Basin2016ContinuousRegulators}, this issue impacts the results on the estimation of the \textit{UBST} of the fixed-time algorithms given in~\cite{Basin2016ContinuousRegulators,Basin2016Finite-andTechniques}. However, recently, Basin et al~\cite{basin2020discussion} provided the arguments to correct such flaws
and showed that the results on the estimation of the \textit{UBST} in~\cite{Basin2016ContinuousRegulators,Basin2016Finite-andTechniques} remain valid. The autonomous controller in~\cite{Basin2016ContinuousRegulators} will be used to illustrate our redesign approach.}). On the other hand, nonautonomous controllers based on a class of time-varying gains known as time-base generators have been derived~\cite{Delfin2016,Song2018,Becerra2018,Song2017}, where the origin is reached exactly at the predefined-time, this feature is referred as prescribed-time~\cite{Song2017,Song2018} or predefined-time~\cite{Becerra2018,Delfin2016}. However, in such methods such as ~\cite{Song2017,Song2018}, to reach the origin, the time-varying gain requires to tend to infinity as the time approaches the prescribed/predefined-time.

To fill these gaps, in this paper, we propose a methodology for the design of stabilizing controllers for a perturbed chain of integrators\footnote{For simplicity, we focus on chains of integrators. However, the results can be straightforwardly extended to a controllable linear system and feedback linearizable nonlinear systems, and further extended to the multivariable case.}. Our approach has the following four properties: 
\begin{enumerate}
    \item  the closed-loop system is fixed-time stable;
    \item  the desired \textit{UBST} is set a priori explicitly, with one parameter;
    \item  the \textit{UBST} can be set arbitrarily tight (i.e., the slack between the predefined and the least \textit{UBST} can be set arbitrarily small);
    \item  the controller is non-autonomous with bounded time-varying gains.
\end{enumerate}

Our methodology consists of redesigning known autonomous controllers by adding time-varying gains, constructed from time-base generators~\cite{Morasso1997}. However, unlike existing methods based on time-base generators\cite{Song2017,Song2018,Pal2020DesignTime}, we provide sufficient conditions such that our time-varying gains remain bounded. Moreover, contrary to~\cite{Becerra2018}, no initial state is explicitly used in the feedback law; and unlike~\cite{Song2017,Becerra2018}, predefined-time convergence is guaranteed even in the presence of external disturbances. 

To illustrate our approach, we show how to redesign the autonomous fixed-time controllers in~\cite{Aldana-Lopez2018} and~\cite{Basin2016ContinuousRegulators} to obtain non-autonomous fixed-time controllers with predefined \textit{UBST}, significantly reducing the over-estimation of the \textit{UBST} while maintaining the time-varying gain bounded\footnote{In the publisher's version, which can be found at \url{https://doi.org/10.1002/rnc.4976}, an example of the redesign of the autonomous controller given in~\cite{Zimenko2018} is also provided.}.

\textbf{Notation:}
$\mathbb{R}$ is the set of real numbers, $\Bar{\mathbb{R}}=\mathbb{R}\cup\{-\infty,+\infty\}$, $\mathbb{R}_+=\{x\in\mathbb{R}\,:\,x\geq0\}$ and $\Bar{\mathbb{R}}_+=\mathbb{R}_+\cup\{+\infty\}$. Finally, $t_0$ denotes the initial time. For a complex number $\lambda$, $\Real{\lambda}$ represents the real part of $\lambda$ and $\Imag{\lambda}$ the imaginary part of $\lambda$. For a time function $\delta(t)$, we write $\delta(t)\equiv0$ to stress that $\delta(t)=0$ for all $t\geq t_0$.

The rest of the manuscript is organized as follows. In Section~\ref{Sec:Prelim}, we present the preliminaries on fixed-time stability and time-scale transformations. In Section~\ref{Sec:MainResult}, we introduce our redesign methodology, which is applied in Section~\ref{Sec:Examples} to redesign a linear controller, and two nonlinear controllers for fixed-time stability (the controller from Aldana-López et al.~\cite{Aldana-Lopez2018} and the controller from Basin et al.~\cite{Basin2016ContinuousRegulators}). Finally, in Section~\ref{Sec:Conclusions}, we present the conclusion and future work.

\section{Preliminaries}
\label{Sec:Prelim}
Consider the system
\begin{equation}\label{eq:sys}
    \dot{x}=-f(x,t)+D\delta(t), \ \forall t\geq t_0, 
\end{equation}
where $x\in\mathbb{R}^n$ is the state of the system, $t\in[t_0,+\infty)$ is time, $D=[0,\ldots,0,1]^T$, and $\delta$ is a disturbance satisfying $|\delta(t)|\leq L$, for a constant $L<+\infty$.

The solutions of~\eqref{eq:sys} are understood in the sense of Filipov~\cite{Cortes2008}. We assume that $f(\cdot,\cdot)$
is such that the origin of~\eqref{eq:sys} is asymptotically stable and, except at the origin, \eqref{eq:sys} has the properties of existence and uniqueness of solutions in forward-time on the interval $[t_0,+\infty)$ (see Proposition~5 in ~\cite{Cortes2008}). 
The set of admissible disturbances, on the interval $[t_0,\hat{t}]$, where $\hat{t}$ is some time satisfying $\hat{t}>t_0$, is denoted by $\mathcal{D}_{[t_0,\hat{t}]}$.
The solution of \eqref{eq:sys} for $t\in [t_0,\hat{t}]$, with disturbance $\delta_{[t_0,\hat{t}]}$ (i.e. the restriction of $\delta(t)$ to $[t_0,\hat{t}]$) and initial condition $x_0$ is denoted by $x(t;x_0,t_0,\delta_{[t_0,\hat{t}]})$, and the initial state is given by $x(t_0;x_0,t_0,\cdot) = x_0$, when $t_0=0$ we simply write $x(t;x_0,\delta_{[t_0,\hat{t}]})$. Moreover, if $\delta(t)\equiv0$ we simply write $x(t;x_0)$.

For simplicity, throughout the paper, we assume that the origin is the unique equilibrium point of the systems under consideration. Thus, without ambiguity, we refer to global stability (in the respective sense) of the origin of the system as the stability of the system. The extension to local stability is straightforward.

\begin{definition}(Settling-time function)
\label{Def:Settling}
The \textit{settling-time function} of system~\eqref{eq:sys} is defined as
$$T(x_0,t_0):=\inf\{\xi\geq t_0:\forall\delta_{[t_0,\infty)}\in\mathcal{D}_{[t_0,\infty)}, \lim_{t\to\xi}x(t;x_0,t_0,\delta_{[t_0,\infty)})=0\}-t_0.$$
\end{definition}

For autonomous systems ($f$ in~\eqref{eq:sys} does not depend on $t$), the settling-time function is independent of $t_0$, in such cases we simply write $T(x_0)$. Notice that, Definition~\ref{Def:Settling} admits $T(x_0,t_0)=+\infty$.

\begin{definition} \label{def:fixed}(Fixed-time stability~\cite{Polyakov2014}) 
System \eqref{eq:sys} is said to be \textit{fixed-time stable} if it is asymptotically stable~\cite{Khalil2002NonlinearSystems} and the settling-time function $T(x_0,t_0)$ is bounded on  $\mathbb{R}^n\times\mathbb{R}_+$, i.e. there exists $T_{\text{max}}\in\mathbb{R}_+\setminus\{0\}$ such that $T(x_0,t_0)\leq T_{\text{max}}$ if $t_0\in\mathbb{R}_+$ and $x_0\in\mathbb{R}^n$. Thus, $T_{\text{max}}$ is a \textit{UBST} of $x(t;x_0,t_0,\delta_{[t_0,\infty)})$.
\end{definition}

%We are interested on finding sufficient conditions on system \eqref{eq:sys} such that a \textit{UBST} is given by the parameter $T_c$, i.e. $T_c=T_{\text{max}}$. 
%Note that there are infinite choices for $T_\text{max}$. Of particular interest is to find sufficient conditions such that $T_c$ is the least \textit{UBST}. Since the \textit{UBST} is set a priori as a parameter of the system we say that \eqref{eq:sys} is fixed-time stable with predefined \textit{UBST}.

\subsection{Time-scale transformations}

As in~\cite{Pico2013,aldana2019design}, the trajectories corresponding to the system solutions are interpreted, in the sense of differential geometry~\cite{Kuhnel2015DifferentialGeometry}, as regular parametrized curves. Since we apply regular parameter transformations over the time variable, then without ambiguity, this reparametrization is sometimes referred to as time-scaling.

\begin{definition} (Definition~2.1 in~\cite{Kuhnel2015DifferentialGeometry})
\label{Def:RegularParamCurve}
A regular parametrized curve, with parameter $t$, is a $C^1(\mathcal{I})$ immersion $c: \mathcal{I}\to \mathbb{R}$, defined on a real interval $\mathcal{I} \subseteq \mathbb{R}$. This means that $\frac{dc}{dt}\neq 0$ holds everywhere.
\end{definition}

\begin{definition}(Pg.~8 in~\cite{Kuhnel2015DifferentialGeometry})
\label{Def:RegularCurve}
A regular curve is an equivalence class of regular parametrized curves, where the equivalence relation is given by regular (orientation preserving) parameter transformations $\varphi$, where $\varphi:~\mathcal{I}~\to~\mathcal{I}'$ is $C^1(\mathcal{I})$, bijective and $\frac{d\varphi}{dt}>0$. Therefore, if $c:\mathcal{I}\to\mathbb{R}$ is a regular parametrized curve and $\varphi:\mathcal{I}\to \mathcal{I}'$ is a regular parameter transformation, then $c$  and  $c\circ\varphi:\mathcal{I}'\to\mathbb{R}$ are considered to be equivalent.
\end{definition}

\section{Main Result}
\label{Sec:MainResult}

Consider the perturbed systems
\begin{align}
 \dot{x}_1&=x_2\\
 &\vdots\\
 \dot{x}_{n-1}&=x_n\\
 \dot{x}_n&=u(t)+\delta(t), \label{Eq:PredefinedSystem}
\end{align}
where $x=[x_1,\ldots,x_n]^T\in\mathbb{R}^n$ is the state, $u(t)$ is the controller and $\delta(t)$ is a bounded disturbance satisfying $|\delta(t)|\leq L, \forall t\geq t_0$, for a known constant $L$.

Let $T_c>0$ be a desired \textit{UBST} for~\eqref{Eq:PredefinedSystem}. Our aim is to present a methodology for designing a stabilizing controllers $u(t)$, such that:
\begin{enumerate}
    \item the closed-loop system~\eqref{Eq:PredefinedSystem} is fixed-time stable;
    \item $T_c$ is a \textit{UBST};
    \item with an appropriate selection of the control parameters, the \textit{UBST} can be made arbitrarily tight (i.e, the slack between the predefined \textit{UBST} and the least \textit{UBST} can be set arbitrarily small). 
    \item the algorithm is non-autonomous with a bounded time-varying gain.
\end{enumerate}

Our methodology consists of defining $u(t)$ as a piecewise controller. First, a non-autonomous controller drives the state to the origin, ensuring that the origin is reached, regardless of the initial condition, in time $t< t_0+T_c$. This controller maintains the state at the origin until time $t=t_0+T_c$ when switching occurs to an autonomous controller designed to maintain the state at the origin despite the disturbance. 

\begin{assumption}
\label{Assum:AssympChain}
The mapping $w_{L}:\mathbb{R}^n\to\mathbb{R}$ is such that with $u(t)=w_{L}(x)$, the perturbed system~\eqref{Eq:PredefinedSystem} is asymptotically stable for all disturbances satisfying $|\delta(t)|\leq L$, for all $t\geq t_0$.
\end{assumption}

Assumption~\ref{Assum:AssympChain} means that we already have a robust controller for system~\eqref{Assum:AssympChain}, but such a controller may not satisfy the real-time constraints. This controller is the one that we will use to maintain the state at the origin for all $t\geq t_0+T_c$. If $L=0$, then $w_{L}(x)$ can be chosen as an appropriate linear state feedback~\cite{Kailath80}. If $L>0$, then $w_{L}(x)$ can be chosen, for instance, as a high order sliding mode control, such as~\cite{Ding2015}. 

Our approach is based on the following time transformation:
\begin{lemma}
\label{Lemma:ParTrans}
Let $\eta$ be a constant satisfying $0<\eta\leq 1$. Then, the function $\varphi(t)=\tau=-\alpha^{-1}\ln(1-\eta(t-t_0)/T_c)$, defines a parameter transformation with $\varphi^{-1}(\tau)=t=\eta^{-1}T_c(1-e^{-\alpha\tau})+t_0$ as its inverse mapping.
\end{lemma}
\begin{proof}
It follows from Definition~\ref{Def:RegularCurve}.
\end{proof}

To derive the controller designed to drive the state of system~\eqref{Assum:AssympChain} to the origin in a fixed-time, with predefined \textit{UBST} given by $T_c$, let us introduce the time-varying gain, which is parametrized by $T_c$ (the desired \textit{UBST}): 
\begin{equation}
\label{Eq:TBG}
    \kappa(t-t_0):=\left\lbrace
    \begin{array}{lll}
      \frac{\eta}{\alpha(T_c-\eta (t-t_0))}   & \text{if} & t\in[t_0,t_0+T_c) \\
       1  &  & \text{otherwise,}
    \end{array}
    \right.
\end{equation}
where $0<\eta\leq1$; together with the following auxiliary system:
\begin{align}
 \frac{dy_1}{d\tau}&=y_2\\
 \frac{dy_2}{d\tau}&=y_3-\alpha y_2\\
 &\vdots\\
 \frac{dy_{n-1}}{d\tau}&=y_n-\alpha (n-2)y_{n-1}\\
 \frac{dy_n}{d\tau}&=\upsilon(y)-\alpha (n-1)y_{n}+\pi(\tau), \label{Eq:TauSyst}
\end{align}
where $y=[y_1,\ldots,y_n]^T$ is the state, $\alpha>0$, $\upsilon(y)$ is the controller, $\tau$ is the new time, associated to $t$ by the time transformation $\tau=-\alpha^{-1}\ln(1-\eta(t-t_0)/T_c)$, and 
\begin{equation}
\pi(\tau):=\left.\left[\kappa(t-t_0)^{-n}\delta(t)\right]\right|_{t=\eta^{-1}T_c(1-e^{-\alpha\tau})+t_0} 
\end{equation}
is a disturbance. 

Notice that the disturbance $\pi(\tau)$ is vanishing, since $|\delta(t)|\leq L$, for all $t\geq t_0$ and 
\begin{equation}
\left.\kappa(t-t_0)^{-1}\right|_{t=\eta^{-1}T_c(1-e^{-\alpha\tau})+t_0}=\alpha\eta^{-1}T_c e^{-\alpha\tau}. 
\end{equation}
Thus, $\pi(\tau)$ is bounded and $\pi(\tau)\to 0$ as $\tau\to+\infty$. Moreover, notice that if $\eta<1$ then~\eqref{Eq:TBG} is bounded.

The controller designed to drive the origin of system~\eqref{Eq:PredefinedSystem} in a fixed-time upper bounded by $t_0+T_c$ consist on redesigning $\upsilon(y)$ with the time-varying gain~\eqref{Eq:TBG}, i.e. 
\begin{equation}
    u(t)=\kappa(t-t_0)^n\upsilon(\Omega(t-t_0)^{-1}x), \text{ for} t\in[t_0,t_0+T_c),
\end{equation} 
where 
\begin{equation}
\Omega(t-t_0)=\mbox{diag}(1,\kappa(t-t_0),\ldots,\kappa(t-t_0)^{n-1})\in\mathbb{R}^{n\times n}.
\end{equation}

To this end, we make the following assumption on the controller $\upsilon(y)$:
\begin{assumption}
\label{Assump:Auxiliary}
The map $\upsilon:\mathbb{R}^n\to\mathbb{R}$ is such that the auxiliary system~\eqref{Eq:TauSyst} satisfies:
\begin{itemize}
    \item the system~\eqref{Eq:TauSyst} is asymptotically stable with settling time function $\mathcal{T}(y_0)$,
    \item $T_{f}$ is the smallest known value such that, for all $y_0\in\mathbb{R}^n$, $\mathcal{T}(y_0)\leq T_{f}\in\Bar{\mathbb{R}}$ (Notice that, if $\sup_{y_0\in\mathbb{R}^n}\mathcal{T}(y_0)=+\infty$ or no upper bound is known for $\sup_{y_0\in\mathbb{R}^n}\mathcal{T}(y_0)$, then $T_{f}=+\infty$).
    \item 
    \begin{equation}
    \label{Eq:LimitSingularity}
    \lim_{t\to t_0+ T_c}\kappa(t-t_0)^{i-1}\left.y_i(\tau;y_0,\pi_{[0,\mathcal{T}(y_0))})\right|_{\tau=-\alpha^{-1}\ln(1-\eta(t-t_0)/T_c)}=0    
    \end{equation} for $i=1,\ldots,n$, where $y_i(\tau;y_0,\pi_{[0,\mathcal{T}(y_0))})$ is the $i$-th element of the solution of~\eqref{Eq:TauSyst}, denoted by $y(\tau;y_0,\pi_{[0,\mathcal{T}(y_0))})$.
    \end{itemize}
\end{assumption}

Notice that, if \eqref{Eq:TauSyst} is finite-time stable, the condition~\eqref{Eq:LimitSingularity} is trivially satisfied.

Now, we are ready to present our main result.

\begin{theorem}
\label{Th:Main}
Let $\kappa(t-t_0)$ be as in~\eqref{Eq:TBG} with $\eta:=(1-e^{-\alpha T_{f}})$ and $T_f$ as defined in Assumption~\ref{Assump:Auxiliary}. Then, if $w_L:\mathbb{R}^n\to\mathbb{R}$ satisfies Assumption~\ref{Assum:AssympChain}, $\upsilon:\mathbb{R}^n\to\mathbb{R}$ satisfies Assumption~\ref{Assump:Auxiliary}, and $u(t)$ is designed as
\begin{equation}
    u(t)=
    \left\lbrace 
    \begin{array}{lll}
        \kappa(t-t_0)^n\upsilon(\Omega(t-t_0)^{-1}x) & \text{if} & t\in[t_0,t_0+ T_c)\\
         w_L(x)& & \text{otherwise},
    \end{array}
    \right.
    \label{Eq:ProposedControl}
\end{equation}
with $\Omega(t-t_0):=\mbox{diag}(1,\kappa(t-t_0),\ldots,\kappa(t-t_0)^{n-1})$, then the system~\eqref{Eq:PredefinedSystem} is fixed-time stable with $T_c$ as the predefined \textit{UBST}.
\end{theorem}
\begin{proof}
Consider the coordinate change $x_i=\kappa(t-t_0)^{i-1}y_i$, $i=1,\ldots,n$, (i.e. $x=\Omega(t-t_0)y$) then $\dot{x}_i=\kappa(t-t_0)^{i-1}\dot{y}_i+(i-1)\dot{\kappa}(t-t_0)\kappa(t-t_0)^{i-2}y_i$ and the dynamics in the new coordinates are 
\begin{align}
    \dot{y}_i&=\kappa(t-t_0)^{1-i}[\kappa(t-t_0)^{i}y_{i+1}-(i-1)\dot{\kappa}(t-t_0)\kappa(t-t_0)^{i-2}y_i]\\
    &=\kappa(t-t_0)[y_{i+1}-(i-1)\dot{\kappa}(t-t_0)\kappa(t-t_0)^{-2}y_i]
\end{align}
for $i=1,\ldots,n-1$ and
\begin{align}
    \dot{y}_n&=\kappa(t-t_0)^{1-n}[-\kappa(t-t_0)^n\upsilon(\Omega(t-t_0)^{-1}x)+\delta(t)-(n-1)\dot{\kappa}(t-t_0)\kappa(t-t_0)^{n-2}y_n]\\
    &=\kappa(t-t_0)[-\upsilon(y)+\kappa(t-t_0)^{-n}\delta(t)-(n-1)\dot{\kappa}(t-t_0)\kappa(t-t_0)^{-2}y_n].
\end{align}

Notice that $\dot{\kappa}(t-t_0)\kappa(t-t_0)^{-2}=\alpha$.

Moreover, consider the parameter transformation given in Lemma~\ref{Lemma:ParTrans}, $\tau=-\alpha^{-1}\ln(1-\eta(t-t_0)/T_c)$, then $t=\eta^{-1}T_c(1-e^{-\alpha\tau})+t_0$ and 
\begin{align}
\left.\frac{dt}{d\tau}\right|_{\tau=-\alpha^{-1}\ln(1-\eta(t-t_0)/T_c)}&=\left. \left[\alpha\eta^{-1} T_c e^{-\alpha\tau}\right]\right|_{\tau=-\alpha^{-1}\ln(1-\eta(t-t_0)/T_c)}\\
&=\alpha\eta^{-1} T_c(1-\eta(t-t_0)/T_c)=\alpha\eta^{-1}(T_c-\eta(t-t_0))
\\
&=\kappa(t-t_0)^{-1}
\end{align}
Thus, the dynamic of the system in the $\tau$ variable is given by~\eqref{Eq:TauSyst}
where $\pi(\tau)=\left.\kappa(t-t_0)^{-n}\delta(t)\right|_{t=\eta^{-1}T_c(1-e^{-\alpha\tau})+t_0}$. Notice that, since $\delta(t)$ satisfies that $|\delta(t)|\leq L$, for all $t\geq t_0$  and
$\left.\kappa(t-t_0)^{-n}\right|_{t=-\eta^{-1}T_c(1-e^{-\alpha\tau})+t_0}=(\alpha \eta^{-1}T_ce^{-\alpha\tau})^n$, then $|\pi(\tau)|\leq (\alpha \eta^{-1}T_c)^n L$ and $\pi(\tau)\to 0$ as $\tau\to+\infty$. 

Let $\Omega(t-t_0)=\mbox{diag}(1,\kappa(t-t_0),\ldots,\kappa(t-t_0)^{n-1})$.
Since, the system~\eqref{Eq:TauSyst} is asymptotically stable with $\mathcal{T}(y_0)$ as its settling time function and $x_0=\Omega(0)y_0$, then~\eqref{Eq:PredefinedSystem} reaches the origin at $$T(x_0,t_0)=\lim_{\tau\to\mathcal{T}(\Omega(0)^{-1}x_0)}\eta^{-1}T_c(1-e^{-\tau})\leq T_c,$$ and the control $w_L(x_1,\ldots,x_n)$ maintains it at the origin for all $t\geq t_0+T(x_0,t_0)$.

Thus,~\eqref{Eq:PredefinedSystem} is fixed-time stable with $T_c$ as the predefined \textit{UBST}.
\end{proof}

\begin{remark}
Notice that system~\eqref{Eq:PredefinedSystem} can be seen as the error dynamics of a closed-loop system. Thus, the results derived in Theorem~\ref{Th:Main}, can be used to design stabilizing algorithms, trajectory tracking controllers, or regulators. 
\end{remark}

\begin{remark}
Without loss of generality, the results in this manuscript can be applied to controllable linear systems and feedback linearizable nonlinear systems~\cite{Isidori2013}. The extension to the multivariable case is straightforward.
\end{remark}

\begin{corollary}
\label{Cor:ConvProp}
The settling time function of system~\eqref{Eq:PredefinedSystem}, under the controller~\eqref{Eq:ProposedControl}, is given by $T(x_0,t_0)=\lim_{\tau\to\mathcal{T}(\Omega(0)^{-1}x_0)}\eta^{-1}T_c(1-e^{-\tau})$. Additionally, the following holds:
\begin{enumerate}
    \item If $\mathcal{T}(y_0)=+\infty$ for every $y_0\neq 0$, then the settling time of~\eqref{Eq:PredefinedSystem} is exactly $T_c$ for every $x_0\neq0$.
    \item If $\sup_{z_0 \in \mathbb{R}^n} \mathcal{T}(y_0)=+\infty$, then $T_c$ is the least \textit{UBST} of~\eqref{Eq:PredefinedSystem}.
    \item If there exist $T_f<+\infty$ such that, for all $y_0\in\mathbb{R}^n$, $\mathcal{T}(y_0)\leq T_{f}$, then $\kappa(t-t_0)$ is bounded for all $t\in[t_0,t_0+T(x_0,t_0))$ and all $x_0\in\mathbb{R}^n$.
\end{enumerate}
\end{corollary}

The next corollary states that, if system~\eqref{Eq:TauSyst} is fixed-time stable with a known \textit{UBST} then the time-varying gain $\kappa(t-t_0)$ is upper bounded, and the predefined \textit{UBST} of~\eqref{Eq:PredefinedSystem}, under the control~\eqref{Eq:ProposedControl}, which is given by $T_c$, can be set arbitrarily tight.

\begin{corollary}
\label{Cor:BoundedGain}
Assume that $T_f= T_{max}^*$ with a known constant $T_{max}^*<+\infty$ (i.e., system~\eqref{Eq:TauSyst} is fixed-time stable with a known \textit{UBST}). Then, $\eta<1$ and $\kappa(t-t_0)$ is upper bounded by
\begin{equation}
\label{Eq:Bound}
\kappa(t-t_0)\leq \left.\frac{\eta}{\alpha(T_c-\eta (t-t_0))}\right|_{t=\eta^{-1}T_c(1-e^{-\alpha T_{max}^*})+t_0}=\frac{\eta}{\alpha T_c(1- \eta)}<+\infty.
\end{equation}
Moreover, let $\hat{T}=\sup_{y_0\in\mathbb{R}^n}\mathcal{T}(y_0)$ and let $s_\alpha$ be the slack (parametrized by $\alpha$) between the least \textit{UBST} of~\eqref{Eq:PredefinedSystem} and the predefined one given by $T_c$, i.e.,  
$$
s_\alpha=\eta^{-1}T_c(1-e^{-\alpha \hat{T}})+t_0-(T_c+t_0)=T_c(1-e^{-\alpha T_{max}^*})^{-1}(1-e^{-\alpha \hat{T}})-T_c.
$$
Then, for every $\epsilon>0$ there exists $\alpha$ such that $s_\alpha\leq \epsilon$.
\end{corollary}

\begin{remark}
\label{Remark:TradeOff}
Notice that, since $\eta$ is a function of the parameter $\alpha$, by tuning $\alpha$, we can select the bound of $\kappa(t-t_0)$. Alternatively, by tuning $\alpha$, we can make the \textit{UBST} arbitrarily tight (i.e., the slack $s_\alpha$ arbitrarily small). However, notice that as $\alpha$ increases, the bound in~\eqref{Eq:Bound} increases, and the slack $s_\alpha$ decreases. Thus, one needs to establish a trade-off between the size of the upper bound for $\kappa(t-t_0)$ and how small the slack $s_\alpha$ is. Clearly, the bigger the bound on the time-varying gain, the higher the tolerance required for simulation, and the more sensitive the controller becomes to measurement noise.
\end{remark}

\begin{example}
A plot of the time scale transformation with $t_0=0$, $\eta=1$ and $T_c=10$ is shown in Figure~\ref{fig:TimeScale}. Notice that $\alpha$ provides a degree of freedom to select how a time in $\tau$ maps to a time in $t$. By Corollary~\ref{Cor:BoundedGain}, by an appropriate selection of $\alpha$, the \textit{UBST} can be set arbitrarily tight.
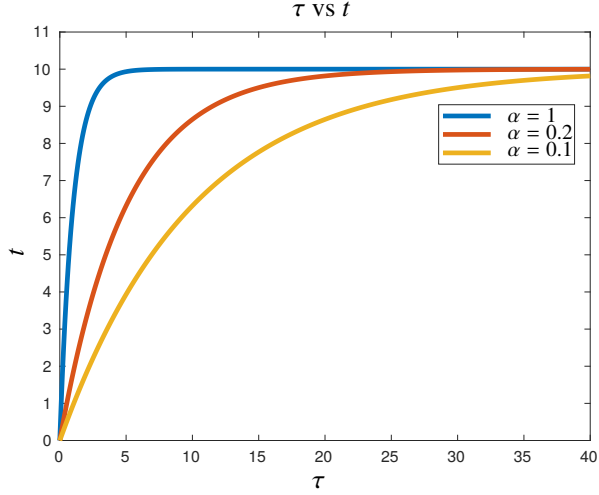
\begin{figure}
    \centering
\def\svgwidth{9cm}
\begingroup%
  \makeatletter%
  \providecommand\color[2][]{%
    \errmessage{(Inkscape) Color is used for the text in Inkscape, but the package 'color.sty' is not loaded}%
    \renewcommand\color[2][]{}%
  }%
  \providecommand\transparent[1]{%
    \errmessage{(Inkscape) Transparency is used (non-zero) for the text in Inkscape, but the package 'transparent.sty' is not loaded}%
    \renewcommand\transparent[1]{}%
  }%
  \providecommand\rotatebox[2]{#2}%
  \newcommand*\fsize{\dimexpr\f@size pt\relax}%
  \newcommand*\lineheight[1]{\fontsize{\fsize}{#1\fsize}\selectfont}%
  \ifx\svgwidth\undefined%
    \setlength{\unitlength}{432bp}%
    \ifx\svgscale\undefined%
      \relax%
    \else%
      \setlength{\unitlength}{\unitlength * \real{\svgscale}}%
    \fi%
  \else%
    \setlength{\unitlength}{\svgwidth}%
  \fi%
  \global\let\svgwidth\undefined%
  \global\let\svgscale\undefined%
  \makeatother%
  \begin{picture}(1,0.73611111)%
    \lineheight{1}%
    \setlength\tabcolsep{0pt}%
    \put(0,0){\includegraphics[width=\unitlength,page=1]{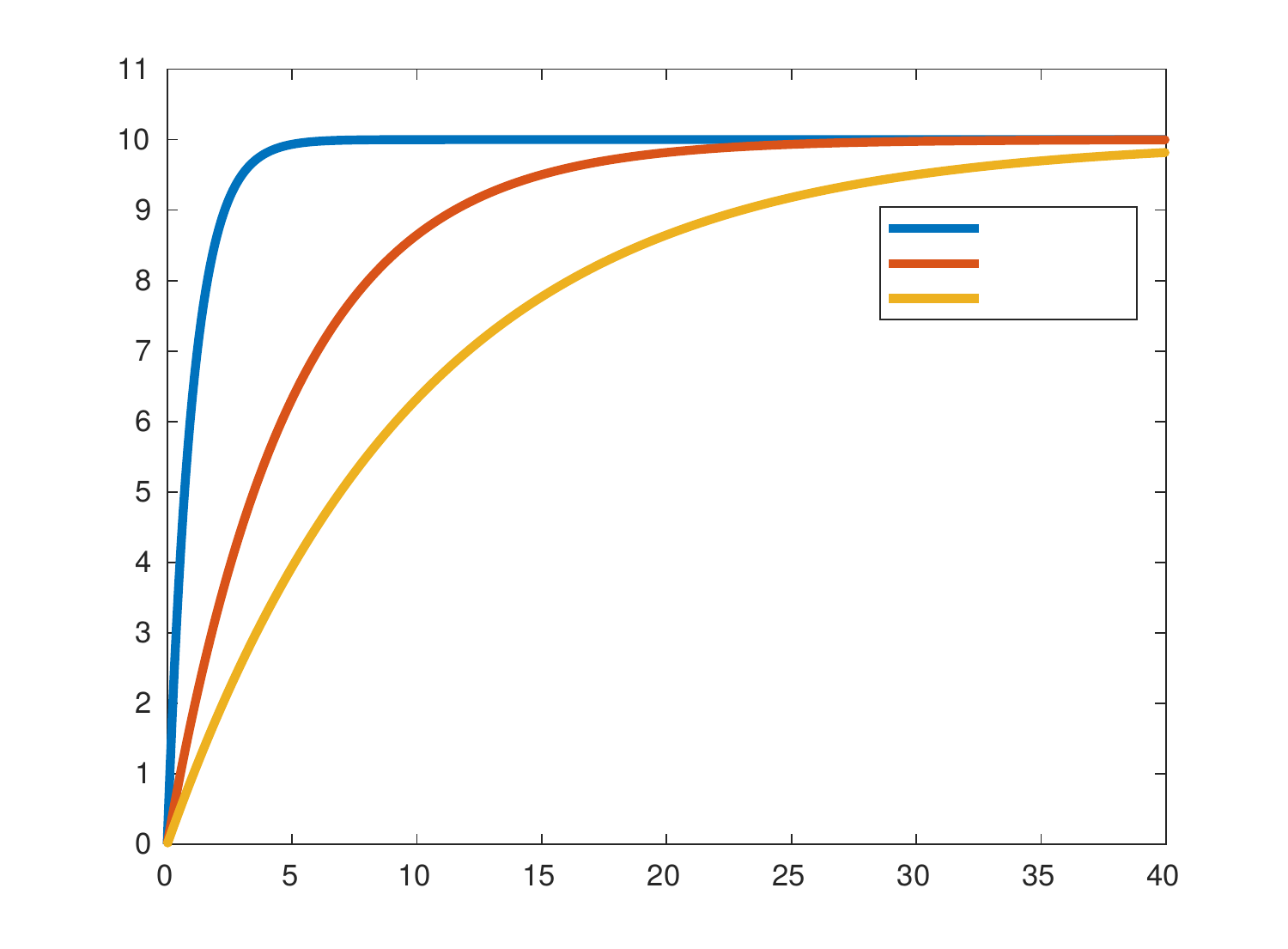}}%
    \put(0.49882888,0.01102061){\color[rgb]{0,0,0}\makebox(0,0)[lt]{\lineheight{1.25}\smash{\begin{tabular}[t]{l}$\tau$\end{tabular}}}}%
    \put(0.07657444,0.35272011){\color[rgb]{0,0,0}\rotatebox{90}{\makebox(0,0)[lt]{\lineheight{1.25}\smash{\begin{tabular}[t]{l}$t$\end{tabular}}}}}%
    \put(0.47,0.70654337){\color[rgb]{0,0,0}\makebox(0,0)[lt]{\lineheight{1.25}\smash{\begin{tabular}[t]{l}$\tau$ vs $t$\end{tabular}}}}%
    \footnotesize{
    \put(0.785,0.55){\color[rgb]{0,0,0}\makebox(0,0)[lt]{\lineheight{1.25}\smash{\begin{tabular}[t]{l}$\alpha=1$\end{tabular}}}}%
    \put(0.785,0.525){\color[rgb]{0,0,0}\makebox(0,0)[lt]{\lineheight{1.25}\smash{\begin{tabular}[t]{l}$\alpha=0.2$\end{tabular}}}}%
    \put(0.785,0.5){\color[rgb]{0,0,0}\makebox(0,0)[lt]{\lineheight{1.25}\smash{\begin{tabular}[t]{l}$\alpha=0.1$\end{tabular}}}}%
    }
  \end{picture}%
\endgroup%
    \caption{Example of a time-scaling with $\eta=1$ and $T_c=10$.}
    \label{fig:TimeScale}
\end{figure}
\end{example}

\section{Examples: Redesigning autonomous controllers to obtain fixed-time non-autonomous system with predefined \textit{UBST}}
\label{Sec:Examples}

Let $a_i$, $i=1,\ldots,n$ be such that $s^n+a_1s^{n-1}+\cdots+a_n=0$ has roots $\lambda_i=\alpha(1-i)$, $i=1,\ldots,n$, and let $Q$ be the similarity transformation taking the linear system $\frac{dy}{d\tau}=Ay+Bu$, where
\begin{equation}
    A=\left[
    \begin{array}{cccccc}
        0 & 1 & 0 &\cdots & 0 & 0 \\
        0 & -\alpha& 1 &\cdots & 0 & 0 \\
        \vdots & \vdots & \vdots & \ddots & \vdots &\vdots \\
        0 & 0 & 0 &\cdots & 1 & 0\\
        0 & 0 & 0 &\cdots & \alpha(n-2) & 1\\
        0 & 0 & 0 &\cdots & 0 & \alpha(n-1)\\
    \end{array}
    \right] \text{ and }
    B=\left[
    \begin{array}{c}
        0 \\
        0 \\
        0\\
        \vdots \\
        0\\
        1\\
    \end{array}
    \right],
\end{equation}
into the controller canonical form~\cite{Kailath80},
i.e. $y=Qz$ where
\begin{equation}
\label{Eq:Tranform}
Q:=\mathcal{C}(A,B)\mathcal{V}
\end{equation}
with 
$\mathcal{C}(A,B)$ as the controllability matrix of the pair $(A,B)$ and
$$
\mathcal{V}=\left[
\begin{array}{ccccc}
a_{n-1} & a_{n-2} & \cdots & a_1 & 1\\ 
a_{n-2} & a_{n-3} & \cdots &  1 & 0\\ 
\vdots & \vdots & \ddots & \vdots & \vdots\\
a_1 & 1 & \cdots & 0 & 0\\
1 & 0 & \cdots & 0 & 0
\end{array}
\right].
$$

\begin{proposition}
\label{Prop:Design}
Let $w_{L_0}(z)$ be such that the system given by
\begin{align}
    \frac{dz_i}{d\tau}&=z_{i+1}\\
    \intertext{for $i=1,\ldots,n-1$, and}
    \frac{dz_n}{d\tau}&=w_{L_0}(z)+\pi(\tau), \label{Eq:Ztau}
\end{align}
where $z=[z_1,\ldots,z_n]^T$, is asymptotically stable with settling-time function given by $\mathcal{T}_z(z_0)$. Then, with
\begin{equation}
\label{Eq:Controlw}
    \upsilon(y)=\left.\left[w_{L_0}(z)+a_nz_1+\cdots+a_1z_n\right]\right|_{z=Q^{-1}y}
\end{equation}
where $a_i$, $i=1,\ldots,n$ are the coefficients of the characteristic polynomial of $A$ and $Q$ is defined in~\eqref{Eq:Tranform}, the system~\eqref{Eq:TauSyst} is asymptotically stable and
the settling time function of~\eqref{Eq:TauSyst} is $\mathcal{T}(y_0)=\mathcal{T}_z(Q^{-1}y_0)$. Thus, under Assumption~\ref{Assum:AssympChain} and~\ref{Assump:Auxiliary}, the controller~\eqref{Eq:ProposedControl} where $\upsilon(y)$ is given by~\eqref{Eq:Controlw}, is fixed-time stable with $T_c$ as the predefined \textit{UBST}.
\end{proposition}
\begin{proof}
Notice that, under the coordinate change $y=Qz$, the system~\eqref{Eq:TauSyst} is given by
\begin{align}
    \frac{dz_i}{d\tau}&=z_{i+1}\\
    \intertext{for $i=1,\ldots,n-1$, and}
    \frac{dz_n}{d\tau}&=\upsilon(Qz)-a_nz_1-\cdots-a_1z_n+\pi(\tau).
\end{align}
Thus, if $\upsilon(y)$ is given by~\eqref{Eq:Controlw}, then, the closed-loop system becomes~\eqref{Eq:Ztau}. Thus, by Assumption~\ref{Assum:AssympChain}, ~\eqref{Eq:Ztau} is asymptotically stable with settling-time function given by $\mathcal{T}_{\eqref{Eq:Ztau}}(z_0)$, where $z_0=Q^{-1}\Omega(\mathbf{0})x_0$. Since $T(x_0,t_0)=\lim_{\tau\to\mathcal{T}_{\eqref{Eq:Ztau}}(z_0)}\eta^{-1}T_c(1-e^{-\alpha\tau})\leq T_c$.
\end{proof}

Notice that, if~\eqref{Eq:Ztau} is finite-time stable (resp. fixed-time stable) then, if $\upsilon(y)$ is given by~\eqref{Eq:Controlw}, the system~\eqref{Eq:TauSyst} is also finite-time stable (resp. fixed-time stable). Moreover, if there exists $T_{max}^*<+\infty$ such that, for all $z_0\in\mathbb{R}^n$, $\mathcal{T}_z(z_0)\leq T_{max}^*$, then for all $y_0\in\mathbb{R}^n$, $\mathcal{T}(y_0)\leq T_{max}^*$. This result is interesting, because it allows to derive non-autonomous controllers with the desired properties based on autonomous controllers from the literature. For instance, fixed-time controllers for~\eqref{Eq:Ztau} have been proposed, for instance, in~\citep{Polyakov2012a,Basin2016ContinuousRegulators,Aldana-Lopez2018,Mishra2018,Zimenko2018}.

\subsection{Deriving fixed-time controllers with predefined \textit{UBST}}

In this subsection, we illustrate our main result by deriving different controllers with predefined \textit{UBST}. These results are derived by applying Proposition~\ref{Prop:Design}. We illustrate our methodology by redesigning three controllers. A linear feedback control~\cite{Kailath80}, the fixed-time controller, proposed in~\cite{Aldana-Lopez2018} and the homogeneous fixed-time controller proposed in~\cite{Zimenko2018}. Additional results can be obtained similarly. The simulations are performed in the OpenModelica software, using the DASSL solver.

The first controller is based on linear feedback control and allows us to obtain a fixed-time stable closed-loop system with settling time exactly at $T_c$. To introduce this result, let us first provide a sufficient condition such that~\eqref{Eq:LimitSingularity} holds for the case when $\upsilon(\cdot)$ is linear.

\begin{lemma}
\label{Lem:ValidLambdas}
Let $\pi(\tau)\equiv 0$ and let \eqref{Eq:TauSyst} be a linear system with eigenvalues $\lambda_i$, $i=1,\ldots,n$. Then, if $\min_{\lambda_i}(|\Real{\lambda_i}|\alpha^{-1})>n$ then~\eqref{Eq:LimitSingularity} holds.
\end{lemma}
\begin{proof}
Notice that $y_i(\tau;x_0)$ is a linear combination of terms of the form:
\begin{itemize}
    \item $e^{\Real{\lambda_i}\tau}$ for distinct $\lambda_i$,
    \item $t^ie^{\Real{\lambda_i}\tau}$, $i=0,\ldots,j-1$ for $\lambda_i$ with algebraic multiplicity $j$ and
    \item $e^{\Real{\lambda_i}\tau}\sin(\Imag{\lambda} \tau+\theta)$ for complex complex conjugate $\lambda_i$ with and $\tan(\theta)=\frac{\Imag{\lambda_i}}{\Real{\lambda_i}}$.
\end{itemize}
However, since 
$$\left.e^{\Real{\lambda_i}\tau}\right|_{\tau=-\alpha^{-1}\ln(1-\eta(t-t_0)/T_c)}=e^{\ln(1-\eta(t-t_0)/T_c)^{|\Real{\lambda_i}|\alpha^{-1}}}=(1-\eta(t-t_0)/T_c)^{|\Real{\lambda_i}|\alpha^{-1}},$$ 
then 
$$\kappa(t-t_0)^{i-1}\left.e^{\Real{\lambda_i}\tau}\right|_{\tau=-\alpha^{-1}\ln(1-\eta(t-t_0)/T_c)}=\eta^{i-1}\frac{(1-\eta(t-t_0)/T_c)^{|\Real{\lambda_i}|\alpha^{-1}}}{(\alpha(T_c-\eta(t-t_0)))^{i-1}}, \ i=1,\ldots,n.$$
Thus, if $\min_{\lambda_i}(|\Real{\lambda_i}|\alpha^{-1})>n$, then~\eqref{Eq:LimitSingularity} holds.
  
\end{proof}

\begin{corollary}
\label{Cor:Linear}
Let $\delta(t)\equiv0$, if $\upsilon(y)$ is given as in~\eqref{Eq:Controlw} with
\begin{equation}
\label{Eq:linear}
w_{L_0}(z)=-k_nz_1-\cdots-k_1z_n
\end{equation}
where $s^n+k_1s^{n-1}+\cdots+k_n$ is a Hurwitz polynomial with roots $\lambda_i$ satisfying $\min_{\lambda_i}(|\Real{\lambda_i}|\alpha^{-1})>n$, then~\eqref{Eq:PredefinedSystem} is fixed-time stable with $T_c$ as the settling time for every nonzero trajectory.
\end{corollary}
\begin{proof}
The proof follows trivially by Lemma~\ref{Lem:ValidLambdas} and Corollary~\ref{Cor:ConvProp}, item 1), and by noticing that~\eqref{Eq:Ztau} is asymptotically stable with eigenvalues $\lambda_i$, $i=1,\ldots,n$.
\end{proof}

\begin{example}
\label{Ex:Linear}
Consider a chain of three integrators and let $\alpha=1$ and $\eta=1$. The simulation of the proposed fixed-time control~\eqref{Eq:ProposedControl} with $\upsilon(y)$ as in~\eqref{Eq:Controlw} with $T_c=10$, $a_1=3$, $a_2=2$, $a_3=0$; $z_1=y_1$, $z_2=y_2$ and $z_3=(y_3-y_2)$, $w_{L_0}(z)$ as in~\eqref{Eq:linear} with $k_1=21$, $k_2=134.75$ and $k_3=257.25$ and $w_L(x)=-6x_1-11x_2-6x_3$ is shown in Figure~\ref{fig:Linear}.

\begin{figure}
    \centering
\def\svgwidth{16.0cm}
\begingroup%
  \makeatletter%
  \providecommand\color[2][]{%
    \errmessage{(Inkscape) Color is used for the text in Inkscape, but the package 'color.sty' is not loaded}%
    \renewcommand\color[2][]{}%
  }%
  \providecommand\transparent[1]{%
    \errmessage{(Inkscape) Transparency is used (non-zero) for the text in Inkscape, but the package 'transparent.sty' is not loaded}%
    \renewcommand\transparent[1]{}%
  }%
  \providecommand\rotatebox[2]{#2}%
  \newcommand*\fsize{\dimexpr\f@size pt\relax}%
  \newcommand*\lineheight[1]{\fontsize{\fsize}{#1\fsize}\selectfont}%
  \ifx\svgwidth\undefined%
    \setlength{\unitlength}{879.92028809bp}%
    \ifx\svgscale\undefined%
      \relax%
    \else%
      \setlength{\unitlength}{\unitlength * \real{\svgscale}}%
    \fi%
  \else%
    \setlength{\unitlength}{\svgwidth}%
  \fi%
  \global\let\svgwidth\undefined%
  \global\let\svgscale\undefined%
  \makeatother%
  \begin{picture}(1,0.21740016)%
    \lineheight{1}%
    \setlength\tabcolsep{0pt}%
    \footnotesize{
    \put(0,0){\includegraphics[width=\unitlength,page=1]{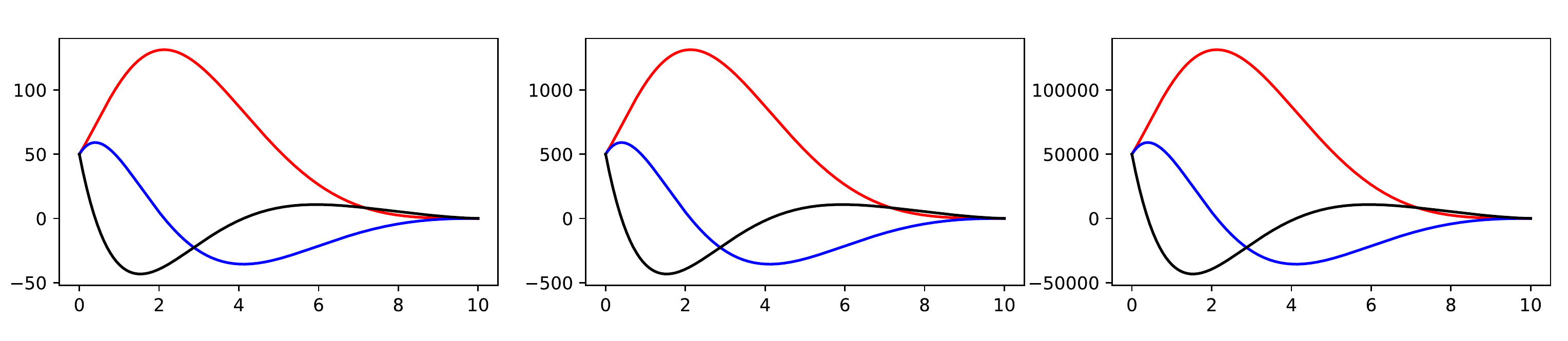}}%
    \put(0.865,0.13){\includegraphics[height=0.044\unitlength,page=1]{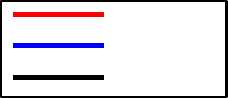}}%
    \put(0.50414292,0.00855746){\color[rgb]{0,0,0}\makebox(0,0)[lt]{\lineheight{1.25}\smash{\begin{tabular}[t]{l}time\end{tabular}}}}%
    \put(0.83996879,0.00855746){\color[rgb]{0,0,0}\makebox(0,0)[lt]{\lineheight{1.25}\smash{\begin{tabular}[t]{l}time\end{tabular}}}}%
    \put(0.16320293,0.00855746){\color[rgb]{0,0,0}\makebox(0,0)[lt]{\lineheight{1.25}\smash{\begin{tabular}[t]{l}time\end{tabular}}}}%
    \put(0.10870009,0.19804402){\color[rgb]{0,0,0}\makebox(0,0)[lt]{\lineheight{1.25}\smash{\begin{tabular}[t]{l}$x_1(0)=x_2(0)=x_3(0)=50$\end{tabular}}}}%
    \put(0.44282128,0.19804402){\color[rgb]{0,0,0}\makebox(0,0)[lt]{\lineheight{1.25}\smash{\begin{tabular}[t]{l}$x_1(0)=x_2(0)=x_3(0)=500$\end{tabular}}}}%
    \put(0.77694248,0.19804402){\color[rgb]{0,0,0}\makebox(0,0)[lt]{\lineheight{1.25}\smash{\begin{tabular}[t]{l}$x_1(0)=x_2(0)=x_3(0)=50000$\end{tabular}}}}%
    \put(0.93,0.163){\color[rgb]{0,0,0}\makebox(0,0)[lt]{\lineheight{1.25}\smash{\begin{tabular}[t]{l}$x_1(t)$\end{tabular}}}}%
    \put(0.93,0.15){\color[rgb]{0,0,0}\makebox(0,0)[lt]{\lineheight{1.25}\smash{\begin{tabular}[t]{l}$x_2(t)$\end{tabular}}}}%
    \put(0.93,0.135){\color[rgb]{0,0,0}\makebox(0,0)[lt]{\lineheight{1.25}\smash{\begin{tabular}[t]{l}$x_3(t)$\end{tabular}}}}%
    }
  \end{picture}%
\endgroup%
    \caption{Simulation of the proposed fixed-time control of Example~\ref{Ex:Linear} with \textit{UBST} chosen as $T_c=10$, which is based on linear control.}
    \label{fig:Linear}
\end{figure}
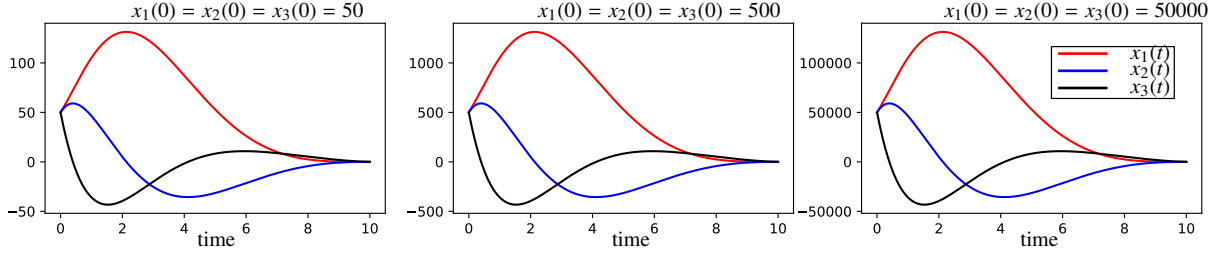

\end{example}

\begin{remark}
Notice that if system~\eqref{Eq:TauSyst} is such that for every $y_0\neq0$, $\mathcal{T}(y_0)=+\infty$ (such as, in the linear case), then $\eta=1$. Thus, $\lim_{t\to t_0+T_c^-}\kappa(t-t_0)=+\infty$. Moreover, notice that, even if $\lim_{t\to t_0+T_c^{-}}u(t)=0$, to compute the control law, one needs to compute the time-varying gain $\kappa(t-t_0)$. Thus, the controller is not realizable in practice. These are drawbacks also present in the non-autonomous controllers proposed in~\cite{Song2017,Song2018,Pal2020DesignTime}; and as stated in~\cite{Song2017}, also the finite-horizon optimal control approach with a terminal constraint, inevitably yields gains that go to infinity. However, unlike such methods, our methodology allows us to design controllers with bounded time-varying gains and the application to perturbed systems, as we illustrate in our next case.
\end{remark}

The second controller is based on the autonomous control for perturbed second order systems proposed in~\cite{Aldana-Lopez2018}, which provides an estimation of the \textit{UBST}. We show that the over-estimation of the \textit{UBST} is significantly reduced in our proposal, while the time-varying gain remains bounded. Compared with previous fixed-time controllers based on time-varying gains, see e.g., ~\cite{Song2017,Pal2020DesignTime}, our approach allows us to guarantee fixed-time convergence with predefined \textit{UBST} and bounded gains even if the system is affected by disturbances.

\begin{corollary}
Assume that $\delta(t)=0$, for all $t\geq t_0$.
Then, if $\upsilon(y)$ is given as in~\eqref{Eq:Controlw} with
\begin{equation}
\label{Eq:BasinScaled}
w_{L_0}(z)=\rho^n(-k_ng_1(z_1)-k_{n-1}g_2(\rho^{-1}z_2)-\cdots-k_1g_n(\rho^{n-1}z_n)),
\end{equation}
where $s^n+k_1s^{n-1}+\cdots+k_n$ is a Hurwitz polynomia, 
\begin{equation}
g_i(x_i)=\lfloor x_i\rceil^{\frac{n-n\varepsilon_1}{n-(i-1)\varepsilon_1}}+\lfloor x_i\rceil^{\frac{n+n\varepsilon_2}{n+(i-1)\varepsilon_2}}
\end{equation} 
with $\varepsilon_1,\varepsilon_2>0$ sufficiently small; 
then, with $\eta=(1-e^{-\alpha T_{max}})$ where $T_{max}=T_{BBF}/\rho$ with $T_{BBF}$ 
an \textit{UBST} of~\eqref{Eq:Ztau} under the control
\begin{equation}
\label{Eq:Basin}
w_{L_0}(z)=-k_ng_1(z_1)-\cdots-k_1g_n(z_n)),
\end{equation}
which is given, for instance, by equation (10) in \cite{Basin2016ContinuousRegulators}; ~\eqref{Eq:PredefinedSystem} is fixed-time stable with $T_c$ as the predefined \textit{UBST}. Moreover, $\kappa(\mathbf{t})$ is bounded for all $t\in[t_0,t_0+T_c]$. 
\end{corollary}
\begin{proof}
It follows from Theorem~2 in~\cite{Basin2016ContinuousRegulators} that the system~\eqref{Eq:Ztau} under the control~\eqref{Eq:Basin} is fixed time stable with \textit{UBST} given by $T_{BBF}$. Now consider the time-scaling $\tau=\rho\hat{\tau}$ together with the coordinate change $\hat{z}_i=\rho^{i-1}z_i$. Since $\frac{dz}{d\hat{\tau}}=\frac{dz}{d\tau}\frac{d\tau}{d\hat{\tau}}$ and $\frac{d\tau}{d\hat{\tau}}=\rho$, then in the $\hat{\tau}$ time variable~\eqref{Eq:Ztau} the dynamic for $\hat{z}$ is given by
\begin{align}
    \frac{d\hat{z}_i}{d\tau}&=\hat{z}_{i+1}\\
    \intertext{for $i=1,\ldots,n-1$, and}
    \frac{d\hat{z}_n}{d\tau}&=\rho^n(-k_ng_1(z_1)-k_{n-1}g_2(\rho^{-1}z_2)-\cdots-k_1g_n(\rho^{n-1}z_n)) \label{Eq:ConScaled}
\end{align}
Since, the \textit{UBST} of~\eqref{Eq:Ztau} under the control~\eqref{Eq:Basin} is upper bounded by $T_{BBF}$, then
an \textit{UBST} of~\eqref{Eq:ConScaled} is given by $T_{max}=T_{BBF}/\rho$. Thus, the result follows trivially.
\end{proof}

\begin{example}
\label{Ex:Basin1}
Consider a chain of three integrators. For comparison, consider $u(t)=w_{L_0}(x)$ with $w_{L_0}(\cdot)$ given by~\eqref{Eq:BasinScaled} with $k_1=3$, $k_2=3$, $k_3=1$, $\varepsilon_1=\frac{3}{22}$, $\varepsilon_2=\frac{3}{18}$, $\rho=1$ which according to~\cite{Basin2016ContinuousRegulators}, $T_{BBF}=578.38$. A simulation of such autonomous system with $T_c=12$ is given in the first two rows of Figure~\ref{fig:Basin1}; it can be observed that, even if $T_{max}=578.38$, convergence is obtained before 50 units of time. Thus, the \textit{UBST} is significantly overestimated.

Now, consider the redesign of this autonomous controller using the proposed method with $\upsilon(y)$ as in~\eqref{Eq:Controlw} with $a_1=3$, $a_2=2$, $a_3=0$; $z_1=y_1$, $z_2=y_2$ and $z_3=(y_3-y_2)$, $w_{L_0}(z)$ as in~\eqref{Eq:BasinScaled} 
where $\rho=578.38/15$; $k_1=3$, $k_2=3$, $k_3=1$, $\varepsilon_1=\frac{3}{22}$, $\varepsilon_2=\frac{3}{18}$; and $w_L(x)=-6x_1-11x_2-6x_3$. A simulation of such closed loop system, with $T_c=65$, $\alpha=1$ and $\eta=1-e^{-15}$ as the parameters for $\kappa(\mathbf{t})$, is given in the last two rows of Figure~\ref{fig:Basin1}. Notice that compared with the autonomous control, the overestimation of the \textit{UBST} is significantly reduced. Additionally, in Figure~\ref{fig:BasinEffort} we present a comparison, between the control signals of the autonomous control and the proposed non-autonomoous control.  Notice that the control signal is lower in the proposed nonautonomous control.

\begin{figure}
    \centering
\def\svgwidth{16cm}    
\begingroup%
  \makeatletter%
  \providecommand\color[2][]{%
    \errmessage{(Inkscape) Color is used for the text in Inkscape, but the package 'color.sty' is not loaded}%
    \renewcommand\color[2][]{}%
  }%
  \providecommand\transparent[1]{%
    \errmessage{(Inkscape) Transparency is used (non-zero) for the text in Inkscape, but the package 'transparent.sty' is not loaded}%
    \renewcommand\transparent[1]{}%
  }%
  \providecommand\rotatebox[2]{#2}%
  \newcommand*\fsize{\dimexpr\f@size pt\relax}%
  \newcommand*\lineheight[1]{\fontsize{\fsize}{#1\fsize}\selectfont}%
  \ifx\svgwidth\undefined%
    \setlength{\unitlength}{909.56811523bp}%
    \ifx\svgscale\undefined%
      \relax%
    \else%
      \setlength{\unitlength}{\unitlength * \real{\svgscale}}%
    \fi%
  \else%
    \setlength{\unitlength}{\svgwidth}%
  \fi%
  \global\let\svgwidth\undefined%
  \global\let\svgscale\undefined%
  \makeatother%
  \begin{picture}(1,0.8010753)%
    \lineheight{1}%
    \setlength\tabcolsep{0pt}%
    \put(0,0){\includegraphics[width=\unitlength,page=1]{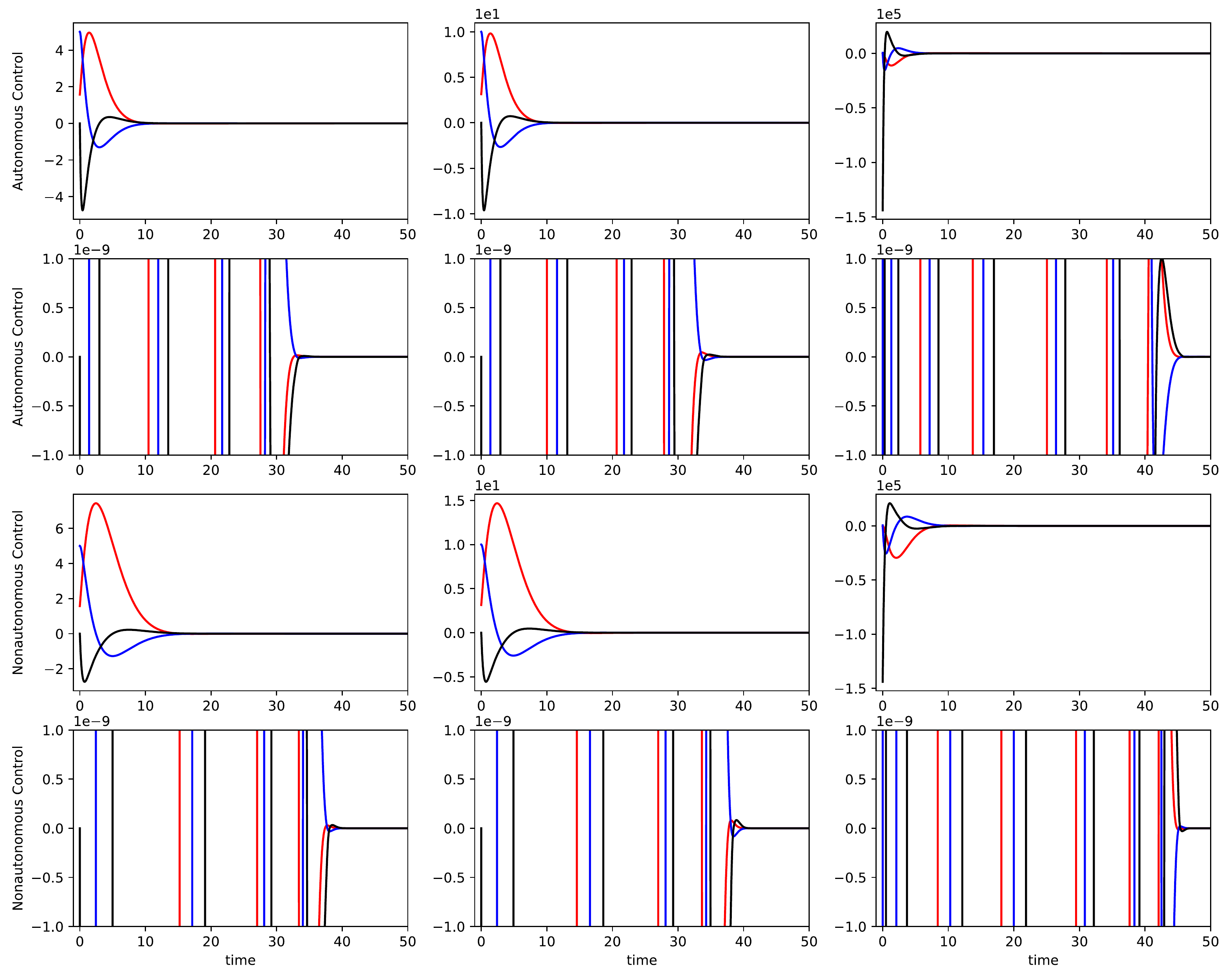}}%
    \put(0.865,0.65){\includegraphics[height=0.05\unitlength,page=1]{Simulations/label.pdf}}%
    \scriptsize{
    \put(0.10,0.785){\color[rgb]{0,0,0}\makebox(0,0)[lt]{\lineheight{1.25}\smash{\begin{tabular}[t]{l}$x_1(0)=\frac{\pi}{2},\ x_2(0)=5,\ x_3(0)=0$\end{tabular}}}}%
    \put(0.42,0.785){\color[rgb]{0,0,0}\makebox(0,0)[lt]{\lineheight{1.25}\smash{\begin{tabular}[t]{l}$x_1(0)=\pi,\ x_2(0)=10,\ x_3=0$\end{tabular}}}}%
    \put(0.75,0.785){\color[rgb]{0,0,0}\makebox(0,0)[lt]{\lineheight{1.25}\smash{\begin{tabular}[t]{l}$x_1(0)=\frac{3\pi}{2},\ x_2(0)=500, x_3=-144e3$\end{tabular}}}}%
    \put(0.925,0.687){\color[rgb]{0,0,0}\makebox(0,0)[lt]{\lineheight{1.25}\smash{\begin{tabular}[t]{l}$x_1(t)$\end{tabular}}}}%
    \put(0.925,0.672){\color[rgb]{0,0,0}\makebox(0,0)[lt]{\lineheight{1.25}\smash{\begin{tabular}[t]{l}$x_2(t)$\end{tabular}}}}%
    \put(0.925,0.657){\color[rgb]{0,0,0}\makebox(0,0)[lt]{\lineheight{1.25}\smash{\begin{tabular}[t]{l}$x_3(t)$\end{tabular}}}}%
    }
  \end{picture}%
\endgroup%
    \caption{Simulation of the proposed fixed-time control of Example~\ref{Ex:Basin1} which is based on the autonomous homogeneous fixed-time control given in~\cite{Basin2016ContinuousRegulators}. Notice that, for the autonomous control the \textit{UBST} is obtained from~\cite{Basin2016ContinuousRegulators} as $578.38$; whereas for the non-autonomous control the desired \textit{UBST} is set at $T_c=65$. Thus, the overestimation is significantly reduced by redesigned the controller with the proposed method. }
    \label{fig:Basin1}
\end{figure}

\begin{figure}
    \centering
\def\svgwidth{16cm}
\begingroup%
  \makeatletter%
  \providecommand\color[2][]{%
    \errmessage{(Inkscape) Color is used for the text in Inkscape, but the package 'color.sty' is not loaded}%
    \renewcommand\color[2][]{}%
  }%
  \providecommand\transparent[1]{%
    \errmessage{(Inkscape) Transparency is used (non-zero) for the text in Inkscape, but the package 'transparent.sty' is not loaded}%
    \renewcommand\transparent[1]{}%
  }%
  \providecommand\rotatebox[2]{#2}%
  \newcommand*\fsize{\dimexpr\f@size pt\relax}%
  \newcommand*\lineheight[1]{\fontsize{\fsize}{#1\fsize}\selectfont}%
  \ifx\svgwidth\undefined%
    \setlength{\unitlength}{898.56433105bp}%
    \ifx\svgscale\undefined%
      \relax%
    \else%
      \setlength{\unitlength}{\unitlength * \real{\svgscale}}%
    \fi%
  \else%
    \setlength{\unitlength}{\svgwidth}%
  \fi%
  \global\let\svgwidth\undefined%
  \global\let\svgscale\undefined%
  \makeatother%
  \begin{picture}(1,0.40774913)%
    \lineheight{1}%
    \setlength\tabcolsep{0pt}%
    \put(0,0){\includegraphics[width=\unitlength,page=1]{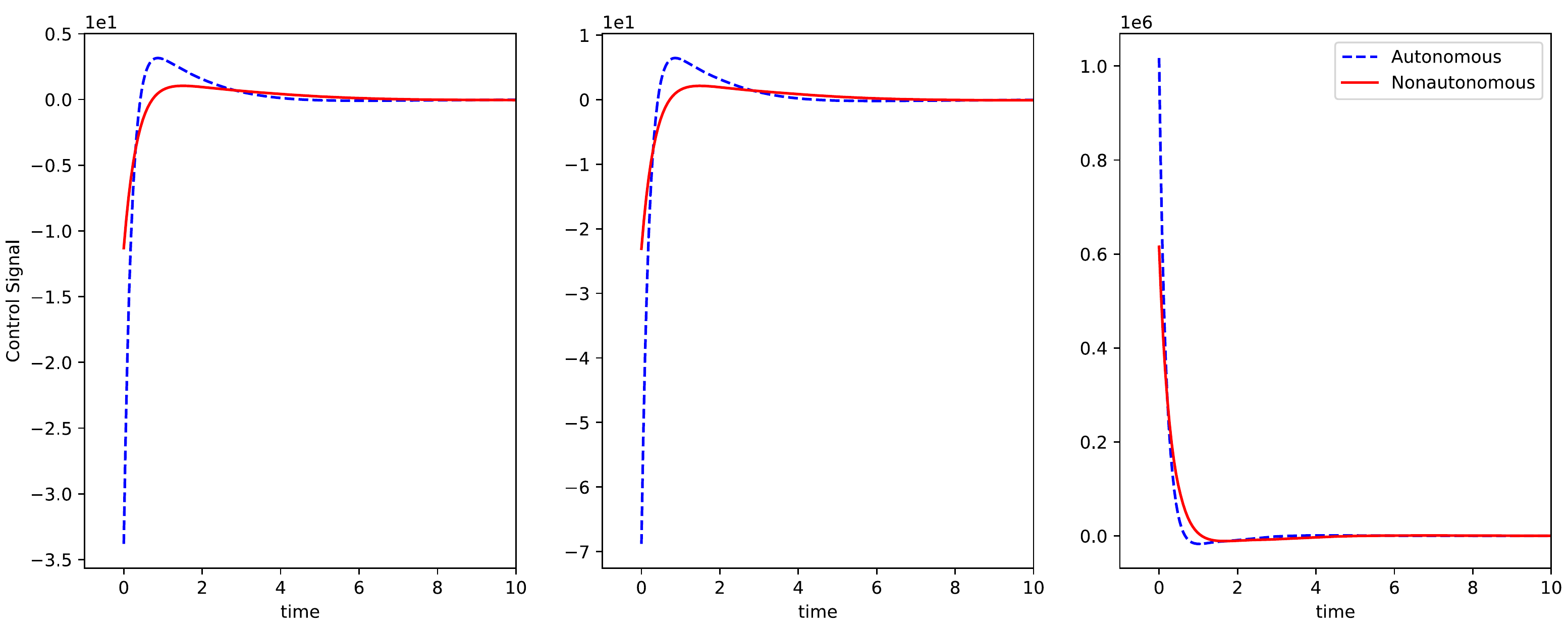}}%
    %\put(0.865,0.2515){\includegraphics[height=0.044\unitlength,page=1]{Simulations/label.pdf}}%
    \scriptsize{
    \put(0.111,0.39){\color[rgb]{0,0,0}\makebox(0,0)[lt]{\lineheight{1.25}\smash{\begin{tabular}[t]{l}$x_1(0)=\frac{\pi}{2},\ x_2(0)=5,\ x_3(0)=0$\end{tabular}}}}%
    \put(0.43,0.39){\color[rgb]{0,0,0}\makebox(0,0)[lt]{\lineheight{1.25}\smash{\begin{tabular}[t]{l}$x_1(0)=\pi,\ x_2(0)=10,\ x_3=0$\end{tabular}}}}%
    \put(0.74,0.39){\color[rgb]{0,0,0}\makebox(0,0)[lt]{\lineheight{1.25}\smash{\begin{tabular}[t]{l}$x_1(0)=\frac{3\pi}{2},\ x_2(0)=500, x_3=-144e3$\end{tabular}}}}%
    \put(0.2,0.34){\color[rgb]{0,0,0}\makebox(0,0)[lt]{\lineheight{1.25}\smash{\begin{tabular}[t]{l}$u(t)$\end{tabular}}}}%
    \put(0.6,0.34){\color[rgb]{0,0,0}\makebox(0,0)[lt]{\lineheight{1.25}\smash{\begin{tabular}[t]{l}$u(t)$\end{tabular}}}}%
    \put(0.85,0.07){\color[rgb]{0,0,0}\makebox(0,0)[lt]{\lineheight{1.25}\smash{\begin{tabular}[t]{l}$u(t)$\end{tabular}}}}%
    }
  \end{picture}%
\endgroup%
    \caption{Control signal for Example~\ref{Ex:Basin1}. Notice that, although the real convergence time are similar in the autonomous and nonautonomous case, the control magnitud is lower in the non-autonomous case. Also, notice that the autonomous control, which is obtained from~\cite{Basin2016ContinuousRegulators}, has an \textit{UBST} of $578.38$; whereas for the non-autonomous control the desired \textit{UBST} is set at $T_c=65$.}
    \label{fig:BasinEffort}
\end{figure}
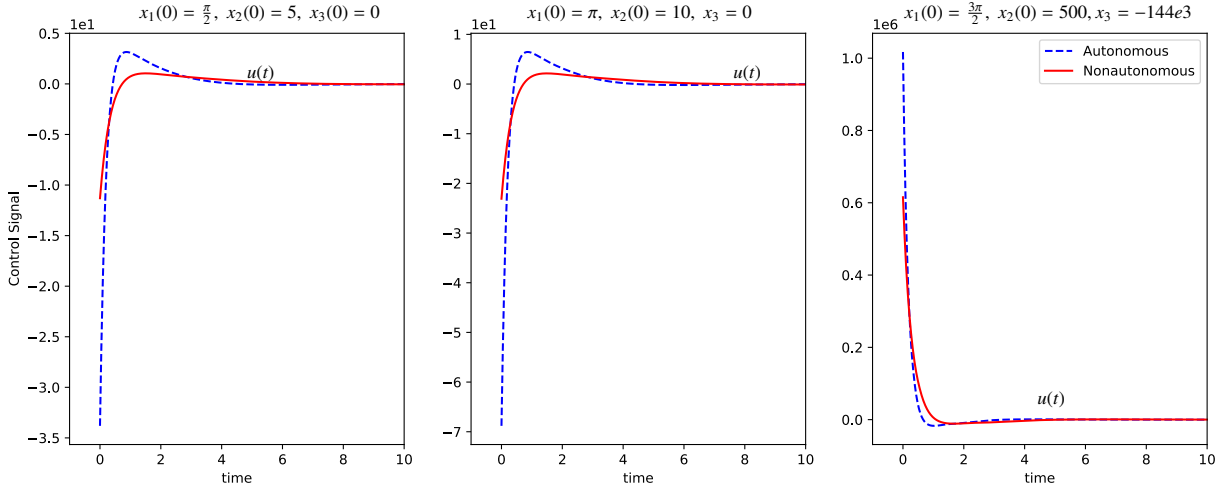
\end{example}

\begin{corollary}
Let $\delta(t)$ be such that $|\delta(t)|\leq L$ holds for a known constant $L$ and let $\alpha_1,\alpha_2,\beta_1,\beta_2,p,q,k>0$, $kp<1$, $kq>1$, $T_{c_1},T_{c_2}>0$, $\zeta(\tau)\geq (\alpha\eta^{-1}T_c e^{-\alpha\tau})^2L$, and \[\gamma_1=\frac{\Gamma \left(\frac{1}{4}\right)^2 }{2\alpha_1^{1/2}\Gamma\left(\frac{1}{2}\right)}\left(\frac{\alpha_1}{\beta_1}\right)^{1/4},\text{ and } \gamma_2=\frac{\Gamma \left(m_{p}\right) \Gamma \left(m_{q}\right)}{\alpha_2^{k}\Gamma (k) (q-p)}\left(\frac{\alpha_2}{\beta_2}\right)^{m_{p}},\] with $m_{p}=\frac{1-kp}{q-p}$ and $m_{q}=\frac{kq-1}{q-p}$. Then, if $\upsilon(y)$ is given as in~\eqref{Eq:Controlw} with
\begin{equation}
\label{Eq:Polyakov}
w_{L_0}(z)=-\left[\frac{\gamma_2}{T_{c_2}}\left(\alpha_2\abs{\sigma}^{p}+\beta_2\abs{\sigma}^{q}\right)^{k}+\frac{\gamma_1^2}{2T_{c_1}^2}\left(\alpha_1+3\beta_1z_1^2\right)+\zeta(\tau)\right]\sign{\sigma},
\end{equation}
where $\sigma$ is defined as
\begin{equation}\label{eq:sigmaso}
\sigma=z_2+\barpow{\barpow{z_2}^2+\frac{2\gamma_1^2}{T_{c_1}^2}\left(\alpha_1\barpow{z_1}^1+\beta_1\barpow{z_1}^3\right)}^{1/2},
\end{equation}
and $\eta$ is selected as $\eta=(1-e^{-\alpha (T_{c_1}+T_{c_2}))})$, then~\eqref{Eq:PredefinedSystem} is fixed-time stable with $T_c$ as the predefined \textit{UBST}. Moreover, $\kappa(\mathbf{t})$ is bounded for all $t\in[t_0,t_0+T_c]$. 
\end{corollary}
\begin{proof}
It follows from Theorem~4 in~\cite{Aldana-Lopez2018}, that the system~\eqref{Eq:Ztau} under the control~\eqref{Eq:Polyakov} is fixed time stable with an \textit{UBST} given by $T_{max}=T_{c_1}+T_{c_2}$. Thus, the proof follows trivially.
\end{proof}

\begin{example}
\label{Ex:Polyakov}
Consider a chain of two integrators with disturbance $\delta(t) = \sin(2\pi t /5)$. For comparison, consider $u(t)=w_{L_0}(x)$ with $w_{L_0}(\cdot)$ given by~\eqref{Eq:Polyakov} where $\zeta=(\alpha\eta^{-1}T_c e^{-\alpha\tau})^2$, $p=0.5$, $q=3$, $k=1.5$ and $\alpha_1=\alpha_2=1/\beta_1=1/\beta_2=4$. It follows from~\cite{Aldana-Lopez2018} that with $T_{c_1} = T_{c_2} = 5$, $T_{max}=10$ is an \textit{UBST}. The simulation under such autonomous control is given in the first two rows of Figure~\ref{fig:Polyakov}.

Now, consider the system under the proposed fixed-time control~\eqref{Eq:BasinScaled} with $\upsilon(y)$ as in~\eqref{Eq:Controlw} with $a_1=1$, $a_2=0$; $z_1=y_1$, $z_2=y_2$, $w_{L_0}(z)$ as in~\eqref{Eq:Polyakov} where $\zeta=1$, $p=0.5$, $q=3$, $k=1.5$ and $\alpha_1=\alpha_2=1/\beta_1=1/\beta_2=4$. It follows from~\cite{Aldana-Lopez2018} that with $T_{c_1} = T_{c_2} = 5$. A simulation of such closed loop system, with $T_c=10$, $\alpha=1$ and $\eta=1-e^{-10}$ as the parameters for $\kappa(\mathbf{t})$, is given in the last two rows of Figure~\ref{fig:Polyakov}. Notice that compared with the autonomous control, the overestimation of the \textit{UBST} is significantly reduced.

\begin{figure}
    \centering
\def\svgwidth{16cm}    
\begingroup%
  \makeatletter%
  \providecommand\color[2][]{%
    \errmessage{(Inkscape) Color is used for the text in Inkscape, but the package 'color.sty' is not loaded}%
    \renewcommand\color[2][]{}%
  }%
  \providecommand\transparent[1]{%
    \errmessage{(Inkscape) Transparency is used (non-zero) for the text in Inkscape, but the package 'transparent.sty' is not loaded}%
    \renewcommand\transparent[1]{}%
  }%
  \providecommand\rotatebox[2]{#2}%
  \newcommand*\fsize{\dimexpr\f@size pt\relax}%
  \newcommand*\lineheight[1]{\fontsize{\fsize}{#1\fsize}\selectfont}%
  \ifx\svgwidth\undefined%
    \setlength{\unitlength}{909.56811523bp}%
    \ifx\svgscale\undefined%
      \relax%
    \else%
      \setlength{\unitlength}{\unitlength * \real{\svgscale}}%
    \fi%
  \else%
    \setlength{\unitlength}{\svgwidth}%
  \fi%
  \global\let\svgwidth\undefined%
  \global\let\svgscale\undefined%
  \makeatother%
  \begin{picture}(1,0.67192635)%
    \lineheight{1}%
    \setlength\tabcolsep{0pt}%
    \put(0,0){\includegraphics[width=\unitlength,page=1]{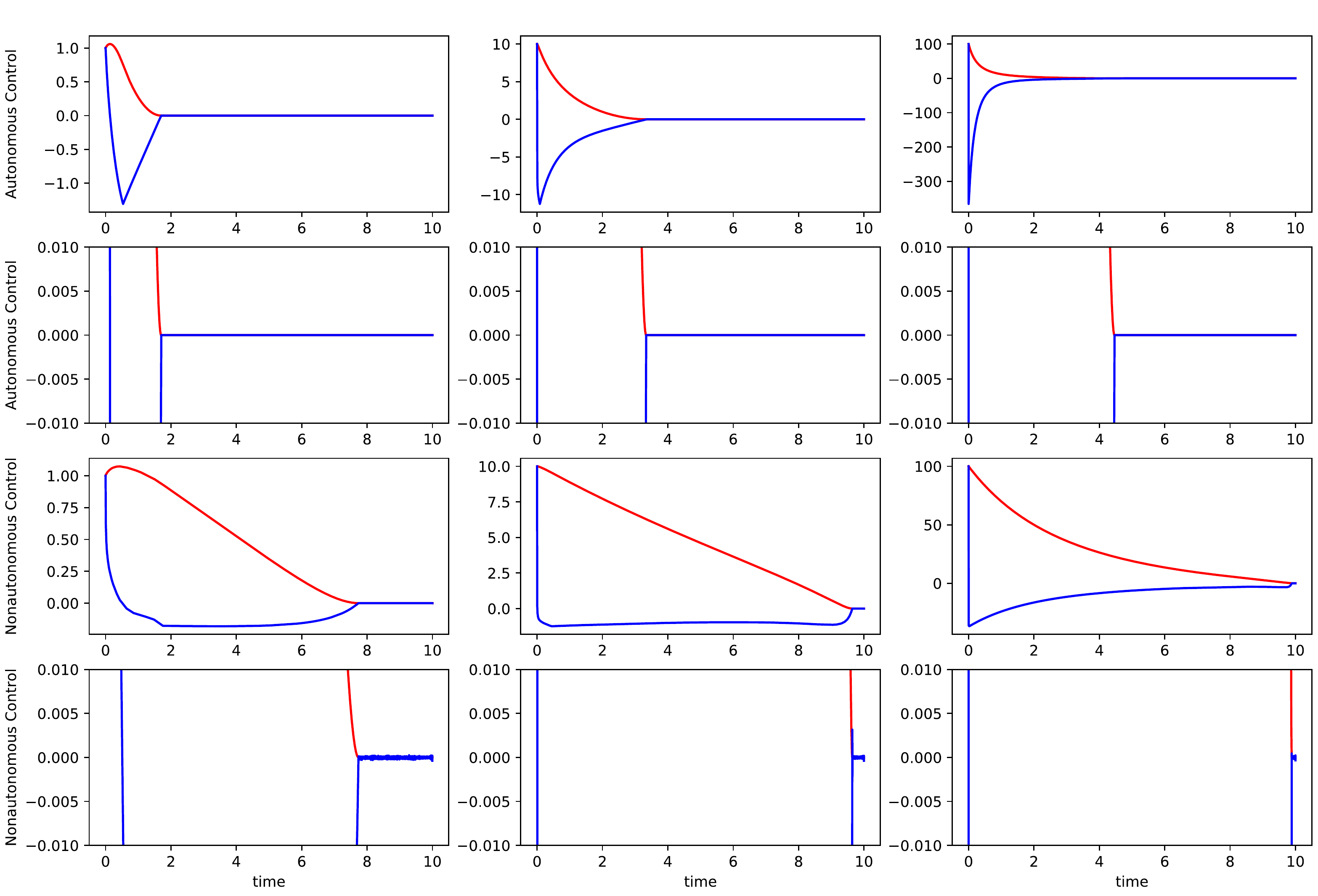}}%
    \put(0.865,0.533){\includegraphics[height=0.031\unitlength,page=1]{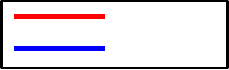}}%
    \footnotesize{
    \put(0.16474359,0.65746217){\color[rgb]{0,0,0}\makebox(0,0)[lt]{\lineheight{1.25}\smash{\begin{tabular}[t]{l}$x_1(0)=x_2(0)=1$\end{tabular}}}}%
    \put(0.49621959,0.65746217){\color[rgb]{0,0,0}\makebox(0,0)[lt]{\lineheight{1.25}\smash{\begin{tabular}[t]{l}$x_1(0)=x_2(0)=10$\end{tabular}}}}%
    \put(0.80955511,0.65746217){\color[rgb]{0,0,0}\makebox(0,0)[lt]{\lineheight{1.25}\smash{\begin{tabular}[t]{l}$x_1(0)=x_2(0)=100$\end{tabular}}}}%
    \put(0.93,0.553){\color[rgb]{0,0,0}\makebox(0,0)[lt]{\lineheight{1.25}\smash{\begin{tabular}[t]{l}$x_1(t)$\end{tabular}}}}%
    \put(0.93,0.54){\color[rgb]{0,0,0}\makebox(0,0)[lt]{\lineheight{1.25}\smash{\begin{tabular}[t]{l}$x_2(t)$\end{tabular}}}}%
    }
  \end{picture}%
\endgroup%
    \caption{Simulation of the proposed fixed-time control of Example~\ref{Ex:Polyakov}, which is based on the second order control given in~\cite{Aldana-Lopez2018}.}
    \label{fig:Polyakov}
\end{figure}
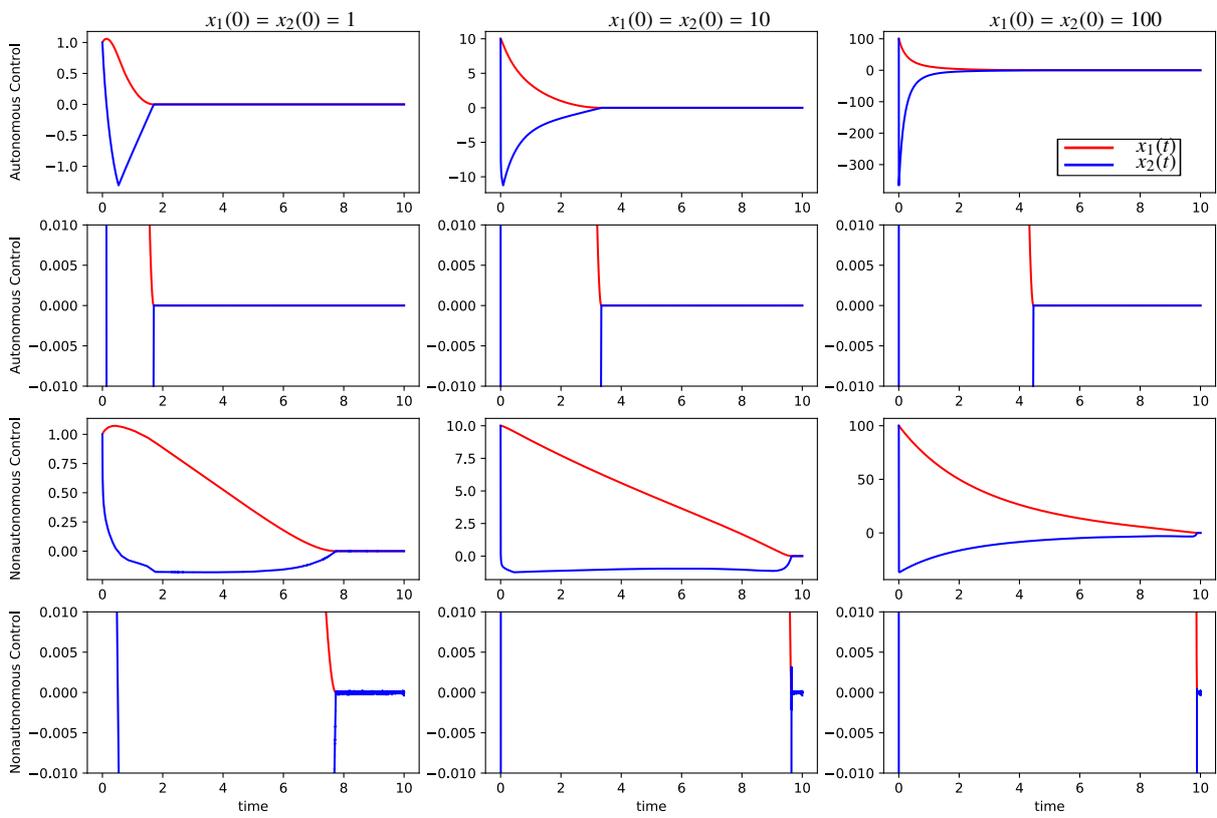
\end{example}

\section{Conclusion}
\label{Sec:Conclusions}
This manuscript presented new controllers based on time-varying gains, constructed from time-base generators, for the stabilization of a chain of integrators in a fixed-time, with predefined \textit{UBST}, which allows the application of the results to scenarios with real-time constraints.

We showed how, based on an autonomous fixed-time controller (which typically has a very conservative \textit{UBST}, see e.g.~\cite{Polyakov2012a,Basin2016ContinuousRegulators,Zimenko2018,Aldana-Lopez2018,Mishra2018}), to derive a non-autonomous fixed-time controller with a predefined \textit{UBST}, resulting in a \textit{UBST} that can be set arbitrarily tight. Additionally, we provide conditions under which the time-varying gain is guaranteed to be bounded, which is a significant contribution against existing controllers based on time-varying gains\cite{Song2017,Song2018,Pal2020DesignTime}. In contrast to existing fixed-time controllers based on time-varying gains\cite{Song2017,Becerra2018,Pal2020DesignTime}, our design is more straightforward and guarantees predefined convergence even in the presence of external disturbances affecting the system. Moreover, unlike~\cite{Becerra2018}, our approach does not require explicit use of the initial conditions in the feedback law. 

We presented numerical simulations and comparisons with existing autonomous fixed-time controllers to demonstrate our contribution. Future work is concerned with the extension of these results to consider a broader class of time-varying gains that have been proposed in the literature~\cite{aldana2019design}.

This work, together with predefined-time observers and online differentiation algorithms~\cite{Aldana-Lopez2020ATime}, contribute toward control algorithms for systems under real-time constraints.

\section*{Acknowledgement}
The author would like to thank Rodrigo Aldana López for the fruitful discussion on fixed-time systems, Nilesh Ahuja for proofreading the manuscript and M. Basin for his comments on autonomous fixed-time systems.

%\bibliographystyle{NJDnatbib}
%\bibliographystyle{abbrv}
%\bibliography{biblio}

\end{document}